\documentclass[12pt]{article}
\usepackage{latexsym,amsmath,enumerate,amssymb,epsf,graphicx, comment,appendix}
\usepackage{mathrsfs}

\newfont{\bb}{msbm10 at 12pt}
\newfont{\tbb}{msbm10 at 8pt}

\def\r{\hbox{\bb R}}

\def\h{\hbox{\bb H}}

\def\s{\hbox{\bb S}}

\def\th{\hbox{\tbb H}}

\def\R{\r}
\def\H{\mathbb{H}}

\def\stop{\hfill$\Box$}

\newcommand{\meta}[2]{\langle #1,#2 \rangle }

\usepackage[latin1]{inputenc}
\usepackage {amsmath}
\usepackage {amsthm}
\usepackage{times}
\usepackage{amscd}
\usepackage{epsf}

\numberwithin{equation} {section}

\begin{document}

\theoremstyle{plain}\newtheorem{lemma}{Lemma}[section]
\theoremstyle{plain}\newtheorem{proposition}{Proposition}[section]
\theoremstyle{plain}\newtheorem{theorem}{Theorem}[section]
\theoremstyle{plain}\newtheorem*{main theorem}{Main Theorem} 
\theoremstyle{plain}\newtheorem{example}{Example}[section]
\theoremstyle{plain}\newtheorem{remark}{Remark}[section]
\theoremstyle{plain}\newtheorem{corollary}{Corollary}[section]
\theoremstyle{plain}\newtheorem*{corollary-A}{Corollary}
\theoremstyle{plain}\newtheorem{definition}{Definition}[section]
\theoremstyle{plain}\newtheorem{acknowledge}{Acknowledgment}
\theoremstyle{plain}\newtheorem{conjecture}{Conjecture}

\begin{center}
\rule{15cm}{1.5pt} \vspace{.4cm}

{\Large \bf Hypersurfaces with Nonegative Ricci Curvature in $\H^{n+1}$} 

\vspace{0.4cm}

{\large Vincent Bonini$\, ^\dag$, Shiguang Ma$\, ^\ddag$\footnote{The author is the corresponding author and
partially supported by NSFC 11571185},
and Jie Qing$\,^\star$\footnote{The author is partially supported by NSF DMS-1608782}}\\

\vspace{0.3cm} 
\rule{15cm}{1.5pt}
\end{center}

\vspace{.2cm}
\noindent$\mbox{}^\dag$ Department of Mathematics, Cal Poly State University, San Luis Obispo, CA 93407; \\e-mail: vbonini@calpoly.edu \vspace{0.2cm}

\noindent$\mbox{}^\ddag$ Department of Mathematics, Nankai University, Tianjin, China; \\e-mail: msgdyx8741@nankai.edu.cn \vspace{0.2cm}

\noindent $\mbox{}^\star$ Department of Mathematics, University of California, Santa Cruz, CA 95064; \\
e-mail: qing@ucsc.edu

\begin{abstract} Based on properties of $n$-subharmonic functions we show that a complete, noncompact, properly embedded hypersurface with 
nonnegative Ricci curvature in hyperbolic space has an asymptotic boundary at infinity of at most two points. Moreover, the presence of two points in the asymptotic 
boundary is a rigidity condition that forces the hypersurface to be an equidistant hypersurface about a geodesic line in hyperbolic space. This gives an affirmative 
answer to the question raised by Alexander and Currier in \cite{AlCu2}.
\end{abstract}


\section{Introduction}\label{Sec:Intro}

For an immersed hypersurface $\phi:M^n \to \h^{n+1}$ with an appropriate orientation, we recall the following successively stronger  pointwise convexity conditions determined by the 
principal curvatures $\kappa_1,\dots,\kappa_n$:
$$
\begin{array}{ll}
\kappa_i >  0 & \text{\em (strict) convexity}\\
\kappa_i (\sum_{j=1}^{n}\kappa_j)  \geq n-1 + \kappa_i^2 & \text{\em nonnegative Ricci curvature}\\
\kappa_i \kappa_j \geq 1 & \text{\em nonnegative sectional curvature}\\
\kappa_i \geq 1 & \text{\em horospherical convexity}
\end{array}
$$
The influence of curvature conditions on the asymptotic boundary of a complete noncompact hypersurface in hyperbolic space $\h^{n+1}$ has been studied in 
\cite{Eps1, Eps2, AlCu,AlCu2, BMQ}. In \cite{Eps2} it is shown that the asymptotic boundary of a complete proper embedding of $\r^2$ into $\h^3$ with nonnegative 
Gaussian curvature has a single point asymptotic boundary. In \cite{AlCu,AlCu2} it is shown that a complete, noncompact, embedded hypersurface $\phi:M^n\to \h^{n+1}$ with 
nonnegative sectional curvature has at most two points in its asymptotic boundary. Moreover the presence of two points in the boundary at infinity is a rigidity condition that 
forces $\phi(M)$ to be an equidistant hypersurface. Recently, in \cite{BMQ} it is shown that the same conclusion as in \cite{AlCu, AlCu2} holds for immersed hypersurfaces.
\\

In \cite{AlCu, AlCu2} it is observed that a properly embedded strictly convex hypersurface in hyperbolic space can be realized as a global vertical graph of a height function 
over a domain in a horosphere and that the height function is subharmonic when restricted to any 2-plane when the hypersurface has nonnegative sectional curvature. Then, based on the theory of subharmonic functions, Alexander and Currier managed to show that the asymptotic boundary is totally disconnected. In \cite{AlCu2} the question was raised as to whether or not nonnegative Ricci curvature suffices for their asymptotic boundary theorem. In this note we affirmatively answer their question.

\begin{main theorem}
For $n\geq 3$, suppose that $\Sigma$ is an $n$-dimensional complete, noncompact hypersurface properly embedded in hyperbolic space $\h^{n+1}$ with nonnegative Ricci curvature. 
Then $\partial_{\infty}\Sigma$ consists of at most two points. The case that $\partial_{\infty}\Sigma$ consists of two points is a rigidity condition that forces $\Sigma$ to be an 
equidistant hypersurface about a geodesic line.
\end{main theorem}

The classification for complete and noncompact Riemannian manifolds with nonnegative Ricci curvature is very interesting and complicated subject (cf. \cite{SS}, for example). 
Our main theorem, on the other hand, classifies those that are properly embedded in hyperbolic space. In fact, our main theorem tells that there are only two classes: 
one is the cylinder $\mathbb{R}\times \mathbb{S}^{n-1}$ and the other consists of nonnegative Ricci curvature metrics on $\mathbb{R}^n$. 

\begin{corollary-A} Suppose that $(M^n, g)$ is a complete and noncompact Riemannian manifold with nonnegative Ricci curvature that can be properly 
isometrically embedded in hyperbolic space $\h^{n+1}$. Then $(M^n, g)$ is either the standard cylinder $\r\times \s^{n-1}$ or it is a complete nonnegative Ricci curvature metric on $\r^n$.
\end{corollary-A}

As suggested for embedded hypersurfaces in \cite{AlCu2}, we realize the rigidity result ultimately as consequence of the Cheeger-Gromoll splitting theorem \cite{CG}, and the Gauss and 
Codazzi equations. In fact, our proof of the rigidity part is local in nature and therefore does not need the embeddedness assumption. The key to our proof is to show that Ricci flat directions are in fact principal directions of the hypersurface for $n\geq 3$ (cf. Lemma \ref{Lem:PrincDir}). To resolve this issue we appeal to the fact that the Ricci operator and the shape operator of a hypersurface in any space form are pointwise simultaneously diagonalizable thanks to \cite{Bour}. 
\\

To show that a connected asymptotic boundary can only be a single point, we pursue an avenue closely related to \cite{AlCu,AlCu2}. We observe that hypersurfaces embedded in 
hyperbolic space with nonnegative Ricci curvature give rise to height functions that are Euclidean $n$-subharmonic. Then we apply the theory of 
$n$-subharmonic functions to show that hypersurfaces embedded in hyperbolic space with nonnegative Ricci curvature must have asymptotic boundaries of Hausdorff dimension zero 
and are therefore a single point when connected. This is not surprising since the analysis of $n$-subharmonic functions in dimension $n$ is somewhat similar to that of subharmonic functions in dimension 2. However, it is rather surprising that our calculation for Ricci curvature in dimensions larger than 2 (cf. Theorem \ref{Lem:n-subharmonic}) 
goes perfectly in line with what was observed in \cite[Theorem 2.1]{AlCu} for Gaussian curvature in dimension 2.
\\

Finally, we echo the question raised in \cite{AlCu2} as to whether or not the Main Theorem in this paper still holds for immersed hypersurfaces.
 

\section{Asymptotic Boundary of Multiple Components}\label{Sec:MultCompBd}

In this section we show that complete noncompact hypersurfaces immersed in hyperbolic space with nonnegative Ricci curvature and multiple component asymptotic boundaries
are in fact equidistant hypersurfaces. Our approach is very much local in nature, hence we do not need to assume the hypersurfaces are embedded. 
\\

Let $(M^n, g)$ be a complete Riemmanian manifold with nonnegative Ricci curvature. If $(M^n,g)$ has a line, then by the Cheeger-Gromoll splitting theorem \cite{CG} 
(see also \cite{Top} for dimension 2), 
$M$ splits isometrically as the product $M \cong {\r} \times N^{n-1}$ where $(N^{n-1}, g_{N})$ is a complete $(n-1)$-manifold with nonnegative Ricci curvature. 
Naturally, the product structure carries to the level of the tangent bundle and the Levi-Civita connection $\nabla$ on $M$, 
which forces the Riemannian curvature tensor of $(M^n,g)$ to split accordingly. Hence, the factor $\R$ of the product $M \cong {\r} \times N$ represents 
a flat direction in $M$.
\\

To be more precise, let $(x_1, x_2,\dots,x_n)$ denote local coordinates on a neighborhood of $M$ adapted to the product structure $M \cong {\r} \times N$ where $x_1 = t$ is 
the coordinate corresponding to distance in the factor $\r$ and $(x_2,\dots,x_n)$ are local coordinates on $N$. Then locally the metric 
$$
g= dt^2+g_{N}
$$ 
where $g_N$ is independent of $t$ and $\frac{\partial}{\partial t}$ is a flat direction. The Riemannian curvature tensor 
\begin{align}\label{Eq:GeodRiemCurv}
R_{ijkt} = 0
\end{align}
for all $i,j,k \in \{1,\dots,n\}$. Therefore,
\begin{equation}\label{Eq:GeodRicciFlat}
R_{it} = 0
\end{equation}
for all $i \in \{1,\dots,n\}$. In other words, the flat direction $\frac{\partial}{\partial t}$ is pointwise an eigendirection for the Ricci curvature operator corresponding to the eigenvalue $0$.
\\

It turns out that the key to establish rigidity is to know that the flat direction is a principal direction of the hypersurface. A pleasantly surprising fact due to Bourguignon \cite{Bour} 
(see also \cite{Besse} Corollary 16.17) is that the Ricci curvature form and the second fundamental form commute since the second fundamental form of a hypersurface in a space form is always a Codazzi tensor.

\begin{lemma}\label{Lem:PrincDir} Suppose that $\phi: M^n\to \mathbb{H}^{n+1}$ is an isometric immersion where $M^n$ has nonnegative Ricci curvature and splits 
as $\r\times N$. Then, for $n\geq 3$, the flat direction is a principal direction for $\phi$.
\end{lemma}

\begin{proof}
It is well-known that the second fundamental form of a hypersurface in a space form is a Codazzi tensor.
Due to Bourguignon \cite{Bour} (see also \cite{Besse} Corollary 16.17), it follows that the Ricci operator and the shape operator then commute. 
Hence, the Ricci operator and the shape operator preserve each other's invariant subspaces and are therefore pointwise simultaneously diagonalizable.  
Let $V_0$ denote the eigenspace of the Ricci operator that corresponds to the eigenvalue $0$ at a point on the hypersurface. Clearly, dim$(V_0) \geq 1$ 
since it contains at least the flat direction. 
\\

Let $\{e_1,\dots,e_n\}$ denote an orthonormal basis of principal directions with the principal curvatures $\kappa_i$ at the point. Up to linear combinations 
of the principal directions in their respective eigenspaces, we may assume that $\{e_1,\dots,e_n\}$ simultaneously diagonalizes the Ricci curvature operator. Moreover, since 
dim$(V_0) \geq 1$, up to reordering we may assume $V_0 = {\text{span}}\{e_1,\dots,e_k\}$ for some $1\leq k \leq n$. Clearly, if $k = 1$, then the flat direction
is a principal direction. Otherwise, let us assume $k \geq 2$.  Then, for each $i=1,\dots, k$,
\begin{equation}\label{Eq:PCRicciFlat}
0=Ric(e_i) = \kappa_i(\sum_{j =1}^{n}\kappa_j)-\kappa_i^2 - (n-1) = \kappa_i H -\kappa_i^2 - (n-1),\vspace{-0.5cm}
\end{equation}
where $\displaystyle H = \sum_{j =1}^{n}\kappa_j$ is the mean curvature. From \eqref{Eq:PCRicciFlat} we see\vspace{-0.3cm}
$$\kappa_i = \frac{H \pm \sqrt{H^2-4(n-1)}}{2} \quad \text{for} \, \, i =1,\dots,k.\vspace{-0.3cm}$$
But then, since $n \geq 3$, $\kappa_i >0$ for all $i = 1, 2, \dots, n$, and $k \geq 2$, we must have
$$\kappa_i =  \kappa_0 = \frac{H - \sqrt{H^2-4(n-1)}}{2} \quad \text{for} \, \, i =1,\dots,k.$$
Therefore, every vector in $V_0 $ is a principal direction associated with the principal curvature $\kappa_0$. 
Thus, the flat direction is a principal direction at any point on the hypersurface. 
\end{proof}

It is interesting to notice that Lemma \ref{Lem:PrincDir} works only for dimensions larger than 2. For flat cases in dimension 2 one needs \cite{VV} instead 
(cf. please see \cite{BMQ} for an alternative proof in dimension 2).  We are now in a position to apply the Codazzi equations to establish the rigidity result.

\begin{theorem} \label{Thm:TwoEndRigidity}
For $n\geq 3$, let $\phi: M^n \to \h^{n+1}$ be an isometric immersion of a complete noncompact manifold $(M^n,g)$ with nonnegative Ricci curvature. If the asymptotic boundary at infinity $\partial_{\infty}\phi(M)$ has more than one connected component, then $\phi(M)$ is an equidistant hypersurface about a geodesic line.
\end{theorem}

\begin{proof}
Let $X_i=\phi_*(\frac{\partial}{\partial x_i})$ denote the local frame on the hypersurface adapted to the product structure. 
From the discussion above we may assume that $X_t=\phi_*(\frac{\partial}{\partial x_1})$ is a unit length flat direction that is 
orthogonal to $X_2,\dots,X_n$. In addition, due to Lemma \ref{Lem:PrincDir}, we may also assume that $X_t$ is a principal direction with 
principal curvature $\kappa_0$. Then, from \eqref{Eq:GeodRiemCurv} and Gauss equations we have
\begin{align} \label{Eq:Gauss}
0 &= R_{itjt} = R^{\mathbb{H}}_{itjt} + II_{ij}II_{tt} - II_{it}II_{tj}  \notag \\
& =- g_{ij} +  \kappa_0 II_{ij} \quad \text{for} \, \, i, j=2,\dots,n
\end{align}
and therefore
\begin{equation}\label{Eq:2ndFormTang}
II = \frac{1}{\kappa_0} g_{N}
\end{equation}
when restricted to directions tangential to $N$. That is, $\kappa_i = \frac 1{\kappa_o}$ for all $i = 2, \dots, n$.
Now since $g=dt^2+g_N$ with $g_N$ independent of $t$, it follows that the Christoffel symbols for $g$ satisfy
\begin{equation}\label{Eq:Christoffel}
\Gamma_{it}^j= \Gamma_{ti}^j=\Gamma_{ij}^t=0 \quad \text{for any} \, \, i,j \in \{1,\dots, n\}.
\end{equation}
Furthermore, from \eqref{Eq:2ndFormTang}, we see
\begin{equation}\label{Eq:PCeq0}
\nabla_{X_t}II_{ii} = X_t(\frac{1}{\kappa_0} (g_N)_{ii}) = || X_i ||_{g_N}^2 X_t(\kappa_i)
\end{equation}
 for any $i \in \{2,\dots, n\}$. Moreover, from the Codazzi equations, we find
\begin{align} \label{Eq:PCeq1}
 \nabla_{X_t}II_{ii}& = \nabla_{X_i}II_{ti} =  -\Gamma_{ii}^l II_{lt} + \Gamma_{ti}^l II_{li} = 0.
\end{align}
Meanwhile,
\begin{equation} \label{Eq:PCeq2}
X_i(\kappa_0)= \nabla_{X_i}II_{tt} = \nabla_{X_t}II_{ti} = -\Gamma_{tt}^l II_{li} + \Gamma_{ti}^l II_{lt} = 0.
\end{equation}
Thus, from \eqref{Eq:PCeq0}, \eqref{Eq:PCeq1} and \eqref{Eq:PCeq2} it follows
the principal curvatures $\kappa_0$ and $\kappa_i=\frac{1}{\kappa_0}$ are constant. 
\\

Due to Currier \cite{Cu}, it follows that $\kappa_0 \neq \kappa_i$ for $i \neq 1$, since otherwise $\kappa_0 = \kappa_i =1$ and 
therefore the hypersurface must be a horosphere, which contradicts the assumption that the hypersurface has more than 
one end. Therefore, locally the hypersurface has exactly two distinct constant principal curvatures $\kappa_0$ of multiplicity 1 
and $\frac{1}{\kappa_0}$ of multiplicity $n-1$.  Then it follows from the classification of so-called isoparametric 
hypersurfaces in hyperbolic space due to Cartan \cite{Cartan}, that the hypersurface is an equidistant hypersurface about a geodesic 
line.
\end{proof}


\section{Calculations for Vertical Graphs in Hyperbolic Space}\label{Sec:AsympBd}

In \cite{Eps2, AlCu, AlCu2} it is observed that a complete, noncompact, properly embedded, strictly convex hypersurface in hyperbolic space can be 
realized globally in Busemann coordinates as a graph of a height function over a domain in a horosphere. Moreover, in \cite{AlCu, AlCu2} it is shown that embedded 
hypersurfaces with nonnegative sectional curvature give rise to height functions that are subharmonic with respect to the Euclidean metric when restricted to any 2-plane. 
Then, as a consequence of the theory of subharmonic functions on domains in the plane, in \cite{AlCu, AlCu2} it is concluded that a hypersurface embedded in hyperbolic 
space with nonnegative sectional curvature must have a single point asymptotic boundary when the asymptotic boundary is connected.
\\

Moving on to the situations when only Ricci curvature is assumed to be nonnegative, the theory of subharmonic functions in dimension 2 is not applicable and the method in 
\cite{AlCu, AlCu2} fails in dimensions larger than 2.  Our approach here is to employ the theory of $n$-subharmonic functions instead of subharmonic functions in dimensions $n>2$.
This approach is not surprising since the analysis of $n$-subharmonic functions in dimension $n$ is somewhat similar to that of subharmonic functions in dimension 2. However, it is rather surprising that our calculation for Ricci curvature in dimensions larger than 2 (cf. Theorem \ref{Lem:n-subharmonic}) 
goes perfectly in line with what was observed in \cite[Theorem 2.1]{AlCu} for Gaussian curvature in dimension 2.
\\

Consider the upper half-space model ${\r}^{n+1}_+$ of hyperbolic space with standard coordinates $(x_1,\dots, x_n, x_{n+1})$ and hyperbolic metric
$$
g_{\th}= \frac{dx_1^2 + \cdots + dx_{n+1}^2}{x_{n+1}^2}.
$$
In the upper half-space model of hyperbolic space we note that
$$
\nabla^{\mathbb{H}}_{\frac{\partial}{\partial x_i}} \frac{\partial}{\partial x_j} = \delta_{ij}\frac{1}{x_{n+1}} \frac{\partial}{\partial x_{n+1}}  
\quad \text{and} \quad \nabla^{\mathbb{H}}_{\frac{\partial}{\partial x_{\alpha}}} \frac{\partial}{\partial x_{n+1}} = -\frac{1}{x_{n+1}} 
\frac{\partial}{\partial x_{\alpha}}.
$$
Note that in our convention Greek letters run from $1, 2, \dots, n+1$ while Latin letters run from $1, 2, \dots, n$.
Let $\Sigma$ be the vertical graph of a function $x_{n+1}=f(x_1,\dots,x_n)$ over a domain $\Omega$ in 
$$
\mathbb{R}^n = \{(x_1, x_2, \dots, x_{n+1})\in {\r}^{n+1}_+: x_{n+1}=0\}. 
$$
Denote the induced tangent vectors on $\Sigma$ by 
$$
X_i = \frac{\partial}{\partial x_i} + f_i \frac{\partial}{\partial x_{n+1}}
$$
where $f_i = \frac{\partial f}{\partial x_i}$. Then the induced metric on $\Sigma$ as a hypersurface in $\h^{n+1}$ is given by
$$
g:= f^{-2}(\delta_{ij} + f_i f_j) dx^i dx^j
$$
with inverse
$$
g^{ij} = f^2(\delta^{ij} - \frac {f_if_j}{1 + |Df|^2}),
$$
where we have denoted the Euclidean norm squared of the Euclidean gradient of $f$ by
$$|D f|^2 = \delta^{ij} f_i f_j = \sum_{i=1}^n f_i^2.$$
Then a straightforward computation gives
$$\nabla^{\mathbb{H}}_{X_i} X_j = f^{-1}((\delta_{ij} + ff_{ij} - f_if_j)\frac{\partial}{\partial x_{n+1}} - f_i\frac{\partial}{\partial x_j} - f_j\frac{\partial}{\partial x_i}).$$
Hence, with respect to unit normal
$$\nu = \frac f{(1+|Df|^2)^\frac 12}(-f_1, -f_2, \cdots, -f_n, 1)$$
on $\Sigma$, we compute the second fundamental form of $\Sigma$ 
\begin{equation}\label{Eq:2ndFForm}
II_{ij} = \meta{\nabla^{\mathbb{H}}_{X_i} X_j}{\nu} = \frac 1{f^2(1+ |Df|^2)^\frac 12} (\delta_{ij} + f_if_j + f f_{ij}).
\end{equation}
Moreover, denoting the Euclidean Laplacian of $f$ by $\Delta f$, it follows that the mean curvature of $\Sigma$ is 
$$
\aligned
H & = \frac 1{(1+|Df|^2)^\frac 12}\sum_{i,j = 1}^n(\delta_{ij} - \frac {f_if_j}{1 + |Df|^2}) (\delta_{ij} + f_if_j + ff_{ij}) \\
& =  \frac 1{(1+|Df|^2)^\frac 12}(n + f\Delta f - \frac {f}{1 +|Df|^2}\sum_{i, j=1}^n f_{ij}f_if_j).
\endaligned
$$
Now at any point $x \in \r^n$  where $h=\log f$ is finite and $Df(x) \neq 0$, we may choose local coordinates where $\frac{\partial}{\partial x_1} = \frac{Df}{|Df|}$ is the Euclidean unit vector in the direction of $Df$ and with $f_j(x) = \frac{\partial f}{\partial x_j}(x)=0$ for all $j \neq 1$. In such coordinates $f_1^2=|Df|^2$ so we may write the mean curvature of $\Sigma$ at such a point $x$ as 
\begin{equation}\label{equ:mean-cur}
H  =  \frac f{(1+ f_1^2)^\frac 32}(f_{11} + \frac {1 + f_1^2}f ) + \frac f{(1+f_1^2)^\frac 12}\sum_{i = 2}^n (f_{ii} + \frac 1f).  
\end{equation}

Next we calculate the Ricci curvature for the vertical graph $\Sigma$ in hyperbolic space via Gauss equations
$$
R^{\Sigma}_{ijkl} = - (g_{ik}g_{jl} - g_{il}g_{jk}) + (II_{ik}II_{jl} - II_{il}II_{jk}).
$$ 
From \eqref{Eq:2ndFForm} it follows that the Ricci curvature tensor has components
\begin{align} \label{Eq:RicciComps}
R_{ik} & = - (n-1)g_{ik} + \frac 1{f^2(1+|Df|^2)}\sum_{j,l=1}^n(\delta_{jl} - \frac {f_jf_l}{1 + |Df|^2}) \notag \\ 
& ((\delta_{ik} + f_if_k +ff_{ik})(\delta_{jl} + f_jf_l + f f_{jl}) 
- (\delta_{il} + f_if_l +ff_{il})(\delta_{jk} + f_jf_k + f f_{jk}) ) \notag \\
& = -(n-1)g_{ik} + \frac 1{f^2(1+|Df|^2)}((\delta_{ik}+f_if_k+ff_{ik})( n + f\Delta f - \frac {f}{1 +|Df|^2} \sum_{j,l=1}^nf_{jl}f_jf_l) \notag \\
& \quad - \sum_{l=1}^n(\delta_{il} + f_if_l +ff_{il})(\delta_{lk} + f f_{lk} -  \frac {f }{1 + |Df|^2} \sum_{j=1}^n f_{jk} f_j f_l )). 
\end{align}

Now, let us consider the gradient of $f$ with respect to the induced metric $g$
$$
\nabla^g f = g^{ij} f_i X_j = f^2(\delta^{ij} - \frac{f_i f_j}{1+|Df|^2})f_i X_j =\frac{f^2}{1+|Df|^2}\sum_{j=1}^n f_j X_j,
$$
and its normalization
$$\frac{\nabla_g f}{||\nabla_g f||_g} =\frac{f}{|Df|(1+|Df|^2)^{\frac{1}{2}}}\sum_{j=1}^n f_j X_j.$$ 
Denoting the components of the normalized gradient of $f$ by
$$
\bar f^i = \frac f{|Df|(1 + |Df|^2)^\frac 12} f_i,
$$ 
from \eqref{Eq:RicciComps} we calculate the Ricci curvature in the direction of the normalized gradient of $f$
\begin{align}\label{Eq:MainCalc1}
R_{ik}\bar f^i\bar f^k & = - (n-1)  + \frac 1{|Df|^2(1+|Df|^2)^2}\sum_{i,k=1}^n f_i f_k \notag \\ 
& \quad((\delta_{ik}+f_if_k+ff_{ik})( n + f\Delta f - \frac {f}{1 +|Df|^2} \sum_{j,l=1}^nf_{jl}f_jf_l) \notag \\
& \quad - \sum_{l=1}^n(\delta_{il} + f_if_l +ff_{il})(\delta_{lk} + f f_{lk} - \frac{f }{1 + |Df|^2} \sum_{j=1}^n f_{jk} f_j f_l )) \notag \\
& = - (n-1)  + \frac 1{|Df|^2(1+|Df|^2)^2}  \\ 
& \quad ((|Df|^2+|Df|^4+f \sum_{i,k=1}^n f_{ik} f_i f_k)( n + f\Delta f - \frac {f}{1 +|Df|^2}\sum_{j,l=1}^n f_{jl}f_jf_l) \notag \\ 
& \quad - \sum_{l=1}^n (f_l + |Df|^2f_l +f \sum_{i=1}^n f_{il} f_i)(f_l + f \sum_{k=1}^n f_{lk} f_k  -  \frac {f}{1 + |Df|^2} \sum_{j,k=1}^n f_{jk} f_j f_l f_k)) \notag.
\end{align}
For convenience, we denote 
$$
\sum_{i,j=1}^nf_{ij}f_if_j = H_1(f) \quad \text{and} \quad \sum_{i,j, k=1}^nf_{ik}f_{kj}f_if_j = H_2(f).
$$
Then we may write \eqref{Eq:MainCalc1} as
\begin{align}\label{Eq:MainCalc2}
R_{ik}\bar f^i\bar f^k &  = - (n-1) +  \frac 1{|Df|^2(1+|Df|^2)^2}(|Df|^2(1 + |Df|^2) + fH_1(f))(n + f\Delta f \notag \\ 
& \quad - \frac {fH_1(f)}{1 + |Df|^2})  - (|Df|^2(1 + |Df|^2) + 2fH_1(f) + f^2 H_2(f) -  \frac {f^2 (H_1(f))^2}{1 +|Df|^2})) \notag \\
& = - (n-1) +  \frac 1{|Df|^2(1+|Df|^2)^2}(n|Df|^2(1 + |Df|^2) + nfH_1(f) \notag \\ 
& \quad + f\Delta f |Df|^2(1 + |Df|^2) + f^2H_1(f)\Delta f  - fH_1(f) |Df|^2 - \frac {f^2(H_1(f))^2}{1 + |Df|^2}  \\
& \quad - (|Df|^2(1 + |Df|^2) + 2fH_1(f) + f^2 H_2(f) -  \frac {f^2 (H_1(f))^2}{1 +|Df|^2})) \notag \\
&  = - (n-1) \frac {|Df|^2}{1+|Df|^2} + \frac f{|Df|^2(1+|Df|^2)^2}((n-2)H_1(f) + \Delta f |Df|^2(1 + |Df|^2) \notag \\ 
& \quad + fH_1(f)\Delta f  - H_1(f) |Df|^2 -  f H_2(f)) \notag.
\end{align}

Now, as above, at any given point where $h=\log f$ is finite and $Df \neq 0$, we choose a local normal 
coordinate such that $\frac {\partial}{\partial x_1}$ is a Euclidean unit vector in the direction of $Df$.
Then pointwise we may simplify \eqref{Eq:MainCalc2} as follows: 
\begin{align}
R_{ik}\bar f^i\bar f^k &  = - (n-1) \frac {f_1^2}{1+ f_1^2} + \frac f{(1+ f_1^2)^2}((n-2)f_{11} + \Delta f (1 + f_1^2) \notag \\ 
&\quad  + ff_{11}\Delta f  - f_{11}f_1^2  -  f \sum_{i=1}^n f_{1i}^2 ) \notag \\
&  = - (n-1) \frac {f_1^2}{1+ f_1^2} + \frac f{(1+ f_1^2)^2}((n-1)(f_{11} + \frac {1+ f_1^2}f) - (n-1) \frac {1+f_1^2}f \notag \\  
& \quad +  f (\frac {1 + f_1^2}f + f_{11}) \sum_{i=2}^n f_{ii}  -  f \sum_{i=2}^n f_{1i}^2 ) \notag \\
&  = \frac {f^2}{(1+ f_1^2)^2}(\frac {1 + f_1^2}f + f_{11}) \sum_{i=2}^n (f_{ii} + \frac 1f)  - (n-1) - \frac {f^2}{(1+ f_1^2)^2}\sum_{i=2}^n f_{1i}^2.
\end{align}
But then, since the Ricci curvature is nonnegative, it follows that $R_{ik}\bar f^i\bar f^k\geq 0$ so
\begin{equation}\label{equ:key}
[\frac f{(1 + f_1^2)^\frac 32}(\frac {1 + f_1^2}f + f_{11})] [\frac f{(1 +f_1^2)^\frac 12} \sum_{i=2}^n (f_{ii} + \frac 1f)] \geq (n-1).
\end{equation}
Note that the sum of the two factors in \eqref{equ:key} is the mean curvature in the light of \eqref{equ:mean-cur}.

\begin{lemma} On a hypersurface in hyperbolic space with nonnegative Ricci curvature the mean curvature of the hypersurface $H\geq n$.
\end{lemma}
\proof From the assumption that the Ricci is nonnegative, for each $i = 1,\dots, n$, one has
$$
\kappa_i H \geq n-1 + \kappa_i^2
$$
where $\kappa_i$ denote the principal curvatures. Therefore, with our choice of orientation, $\kappa_i > 0$ and
$$
H^2 \geq n(n-1) + \sum_{i=1}^n \kappa_i^2 \geq n(n-1) + \frac 1n H^2
$$
which implies that $H\geq n$.
\endproof
 
Since both sum and product are positive, the two factors on the left of the equation \eqref{equ:key} are both positive. Therefore, 
\begin{equation}\label{Eq:Main}
\sqrt{(n-1)(\frac {1 + f_1^2}f + f_{11})}\cdot\sqrt{\sum_{i=2}^n (f_{ii} + \frac 1f)} \geq (n-1)\frac {1 + f_1^2}f.
\end{equation}

\begin{theorem} \label{Lem:n-subharmonic}
Suppose that $\Sigma$ is a vertical graph of a function $x_{n+1}=f(x_1,\dots,x_n)$ in the upper half-space model of hyperbolic space with $f \in C^2$ wherever the hyperbolic height function $h=\log f$ is finite. If $\Sigma$ has nonnegative Ricci curvature, then the height function is Euclidean $n$-subharmonic. That is, 
\begin{equation}\label{Eq:n-subharmonic}
\Delta_n \log f = \text{Div}(|D\log f|^{n-2} D\log f) \geq 0
\end{equation}
wherever $h=\log f$ is finite.
\end{theorem}

\proof 
One may focus on the points where $Df \neq 0$. From \eqref{Eq:Main} and Young's inequality, we have
$$
2(n-1) \frac {1 + f_1^2}{f^2} \leq (n-1)(\frac {f_{11}}f + \frac {1+f_1^2}{f^2}) + \sum_{i=2}^n \frac {f_{ii}}f + (n-1)\frac 1{f^2}
$$
which implies
\begin{align}
0 & \leq (n-1) \frac {f_{11}}f  - (n-1) \frac {f_1^2}{f^2} + \sum_{i=2}^n \frac {f_{ii}}f = (n-1) (\log f)_{11} + \sum_{i=2}^2 (\log f)_{ii} \notag \\
& = (n-2) |D\log f|^{-2}\sum_{i,j=1}^n (\log f)_{ij}(\log f)_i(\log f)_j + \Delta (\log f) \notag \\
& =  |D\log f|^{-(n-2)} \Delta_n \log f
\end{align}
and completes the proof.
\endproof


\section{n-Subharmonic Functions and Proof of Main Theorem}\label{Subsec:n-HarmFncts}

Let $\Sigma$ be a complete, noncompact, properly embedded hypersurface
in $\h^{n+1}$ with nonnegative Ricci curvature. It is known that $\Sigma$ is the boundary of a strictly convex body $\mathbb{U}$ in hyperbolic space. 
Then the recession set $R(\Sigma)$ for $\Sigma$ is the collection of end points at infinity of all geodesic rays which lie entirely inside $\mathbb{U}$. 
Thanks to \cite{Eps2, AlCu, AlCu2}, $\Sigma$ is also known to be a vertical graph of a height function 
over a domain in any horosphere centered at a point in the recession set (cf. \cite{Eps2, AlCu, AlCu2}). 
Let us state a lemma to collect some useful facts for us.

\begin{lemma} (cf. \cite[Proposition 2.2]{AlCu}) \label{Lem:facts-height} Suppose that $\Sigma$ is a complete, noncompact, properly embedded, strictly convex hypersurface in hyperbolic space. Then $\Sigma$ is a graph of
a height function $h:\Omega\to \r\bigcup\{-\infty\}$. Moreover the following hold:
\begin{itemize}
\item The domain $\Omega$ is convex and open in a horosphere centered at some point $p_0$ in the recession set $R(\Sigma)$.

\item The height function $h$ is continuous and locally bounded from above in $\Omega$.

\item $\mathbb{P}(\{h = -\infty\})\bigcup\{p_0\}$ is the recession set $R(\Sigma)$, where $\mathbb{P}$ is the simple orthogonal projection when using the half space model taking $p_0$ as
the infinity.
\end{itemize}
\end{lemma}

\noindent
The facts in Lemma \ref{Lem:facts-height} are based on the observation that all geodesic lines from the point $p_0$ are exactly and exclusively of three kinds: those lying entirely inside $\mathbb{U}$; those lying entirely outside $\mathbb{U}$; those intersecting $\Sigma$ transversally.  Notice that those geodesic lines lying inside or outside can not touch $\Sigma$.
The end points of those geodesic lines inside $\mathbb{U}$ or intersecting $\Sigma$ make up the domain $\Omega$, while the rest is $\mathbb{P}(\{h=\infty\})$. 
It is easily seen that $\Omega$, which is the vertical projection of the convex body bounded by the graph, is convex and the set of all those geodesic lines outside 
$\mathbb{U}$ is closed, that is to say, $\Omega$ is open. For more details, readers are referred to the proof of \cite[Proposition 2.2]{AlCu}. 
In this section, based on the theory of $n$-subharmonic functions and $n$-polar sets in \cite{HKM, Lind}, we present an argument here to show that 
for a complete properly embedded hypersurface with nonnegative Ricci curvature, the set 
$$
\{h = -\infty\}\subset \Omega
$$
is totally disconnected, which implies the main theorem in this paper.
\\

For the convenience of the readers we recall some of the basics in the theory of $p$-subharmonic functions on domains in $\r^n$. Our introduction here is mostly based
on \cite{HKM, Lind}, therefore readers are referred to \cite{HKM, Lind} for details and proofs. First we recall Definition 7.1 of \cite{HKM} (see also Definition 5.1 of \cite{Lind}), 
which defines viscosity $p$-subharmonic functions in terms of the comparison principle.

\begin{definition} (\cite[Definition 7.1]{HKM} \cite[Definition 5.1]{Lind}) \label{Def:Psubharm}
A function $u: W \rightarrow \r \cup \{-\infty\}$ is called viscosity $p$-subharmonic in a domain $W\subset \r^n$, if

1) $u$ is upper semi-continuous in $W$;

2) $u \not\equiv-\infty$ in $W$;

3) For each $W_{1}\subset\subset W$, the comparison principle holds: if $v \in C(\overline{W}_{\hspace{-0.05cm}1})$ is $p$-harmonic in $W_{1}$ and
$v|_{\partial W_{1}} \geq u|_{\partial W_{1}}$, then $v \geq u$ in $W_{1}$.
\end{definition}

\noindent
The most important analytic tools for us are Theorems 10.1 and 2.26 in \cite{HKM}, which we state as follows: 

\begin{theorem} (\cite[Theorems 10.1 and 2.26]{HKM}) \label{Thm:main-tool}
Suppose that $u$ is a viscosity $p$-subharmonic function defined in a domain $W \subset \r^n$. Then its $p$-polar set $\{u = -\infty\}\subset W$ is of 
Hausdorff dimension at most $n-p$. Particularly, for a viscosity $n$-subharmonic function $u$, the set $\{u=-\infty\}$ is of zero $n$-capacity and 
$$
\text{dim}_{\mathscr{H}}(\{u=-\infty\}) = 0.
$$  
\end{theorem}

\noindent
Therefore, the main issue in proving the Main Theorem is to verify that the height functions for complete properly embedded hypersurfaces in hyperbolic space 
with nonnegative Ricci curvature are viscosity $n$-subharmonic in $\Omega \subset\r^n$. In the light of Definition \ref{Def:Psubharm}, we only need to verify the comparison principle. For this purpose we introduce the notion of weakly $p$-subharmonic functions.

\begin{definition} (\cite[Definition 2.12]{Lind}) For $p\geq1$ and a domain $W\subset \r^n$, a function
$u \in W_{loc}^{1,p}(W) $ satisfying 
\begin{equation}
\int \meta {|D u|^{p-2}D u}{D \eta} dx \leq 0 \quad \text{for each} \,\,\eta \in C_{0}^{\infty}(W) \ \text{and $\eta \geq 0$}
\end{equation}
is called a weakly $p$-subharmonic function in $W$.
\end{definition}

\noindent
From Theorem 2.15 in \cite{Lind} and subsequent remarks we have the following comparison principle for weakly $p$-subharmonic functions.

\begin{theorem}(\cite[Theorem 2.15]{Lind}) \label{Lem:WeakCompPrinc}
Suppose that $u$ is a weakly $p$-subharmonic function and $v$ is a $p$-harmonic function in a bounded domain $W \subset \r^n$. If for every $\zeta \in \partial W$
\begin{equation}
\limsup_{x\rightarrow\zeta}u(x)\leq\liminf_{x\rightarrow\zeta}v(x)
\end{equation}
with the possibilities $\infty \leq \infty$ and  $-\infty \leq -\infty$  excluded, then $u\leq v$ almost everywhere in $\Omega$.
\end{theorem}

\noindent
Consequently, due to Theorem \ref{Lem:n-subharmonic} in the previous section, away from the recession set, the height function $h$ is clearly weakly 
$n$-subharmonic and satisfies the comparison principle. Now we are ready to prove our main theorem.
\\

\noindent{\it Proof of the Main Theorem}: \quad
We claim that the height function $h=\log f$ is viscosity $n$-subharmonic in its domain $\Omega$ as defined in Lemma \ref{Lem:facts-height}. 
It is clear that $h\not\equiv-\infty$ and that $h$ is upper semi-continuous. One only needs to verify the Comparison
Principle in $3)$ of  Definition \ref{Def:Psubharm}. Assume otherwise, that condition $3)$ does not hold for $h$ in $\Omega$.  
Let $v \in C(\overline{W})$ be an $n$-harmonic function in $W\subset\subset \Omega$ with $v \geq h$ on $\partial W$ but $h>v$ in some nonempty open subset $W_0 \subset W$ with 
$h=v$ on $\partial W_{0}$. Then it is easily seen  that $W_0 \cap \{h = -\infty\} = \emptyset$. That is to say the height function $h$ is finite in 
$W_0$ and therefore satisfies the comparison principle Theorem \ref{Lem:WeakCompPrinc} on $W_0$, which is a contradiction. 
Thus, the height function is indeed viscosity $n$-subharmonic.
\\

In the light of Theorem \ref{Thm:main-tool}, we know that $\text{dim}_{\mathscr{H}}(R(\Sigma)) =0$. So 
the asymptotic boundary is totally disconnected, that is, every connected component of the asymptotic boundary can only be a single point. 
If the asymptotic boundary has more than two points,  then we know $\partial_{\infty}\Sigma$ consists of exactly two points by the 
Cheeger-Gromoll splitting theorem \cite{CG} and the discussion in Section \ref{Sec:MultCompBd}. Hence, 
by Theorem \ref{Thm:TwoEndRigidity}, it follows that $\Sigma$ is an equidistant hypersurface. Otherwise, the asymptotic boundary 
must consist of a single point. So the proof of the Main Theorem is complete.\stop



\begin{thebibliography}{999999}

\footnotesize
\bibitem{AlCu} S. Alexander and R. J. Currier, {\it Nonnegatively curved hypersurfaces of hyperbolic space and subharmonic functions}, {\it J. London Math. Soc.} {\bf 41} (2) (1990), 347--360.

\bibitem{AlCu2} S. Alexander and R. J. Currier, {\it Hypersurfaces and nonnegative curvature}, {\it Proceedings of Symposia in Pure Mathematics} {\bf 54} (3) (1993), 37--44.


\bibitem{Besse} A. L. Besse, {\it Einstein Manifolds}, Springer-Verlag, Berlin 1987.



\bibitem{BMQ} V. Bonini, S. Ma, and J. Qing, {\it On nonnegatively curved hypersurfaces in hyperbolic space},\\
{\it arXiv: 1603.03862}.

\bibitem{Bour} J.P. Bourguignon, {\it Les vari\'{e}t\'{e}s de dimension 4 \`{a} signature non nulle dont la courbure est harmonique sont d'Einstein}, {\it Invent. Math.}, {\bf 63} (1981), 263--286.


\bibitem{Cartan} E. Cartan, {\it Familles de surfaces isoparam\'{e}triques dans les espaces \`{a} courbure constante}, {\it Ann. Mat. Pura Appl.} (4) {\bf 17} (1938), 177--191.

\bibitem{CG} J. Cheeger and D. Gromoll, {\it The splitting theorem for manifolds of nonnegative Ricci curvature}, {\it J. Diff. Geom.} {\bf 6} (1971), 119--128.


\bibitem{Cu} R. J. Currier, {\it On surfaces of hyperbolic space infinitesimally supported by horospheres}, {\it Trans. Amer. Math. Soc.} {\bf 313} (1)  (1989), 419--431.





\bibitem{Eps1} C.L. Epstein, {\it Envelopes of horospheres and Weingarten surfaces in hyperbolic 3-space}, {\it Unpublished} (1986). {\rm http://www.math.upenn.edu/~cle/papers/index.html}.

\bibitem{Eps2} C.L. Epstein, {\it The asymptotic boundary of a surface imbedded in $\h ^3$ with nonnegative curvature}, {\it Michigan Math. J.} {\bf 34} (1987),  227--239.



\bibitem{HKM} J. Heinonen, T. Kilpelainen, and O. Martio, {\it Nonlinear
potential theory of degenerate elliptic equations}, Oxford Univ. Press,
Oxford, 1993.


\bibitem{Hu} A. Huber, {\it On subharmonic functions and differential geometry in the large},
{\it Comment. Math. Helv.} {\bf 32} (1957), 13--72.


\bibitem{Lind} P. Lindqvist, {\it Notes on the p-Laplace equation},  {\it University of Jyvaskyla Lecture Notes}, 2006.

\bibitem{Top} V. A. Toponogov, {\it Riemannian spaces which contain straight lines}, {\it Amer. Math. Soc. Transl.} {\bf 37} (2) (1964), 287--290. 





\bibitem{SS} Z. Shen and C. Sormani, {\it The topology of open manifolds with nonnegative Ricci curvature}, {\it Commun. Math. Anal.}  2008, {\bf Conference 1}, 20 - 34.





\bibitem{VV}  Yu. A. Volkov and S. M. Vladimirova, {\it Isometric immersions in the Euclidean plane in Lobachevskii space},  
{\it Math. Zametki}  {\bf 10} (1971), 327--332; {\it Math. Notes} {\bf 10} (1971), 619-622.


\end{thebibliography}
\end{document}